\documentclass[reqno]{amsart}
\usepackage[english]{babel}
\usepackage{amscd,amssymb,amsmath,amsfonts,latexsym,amsthm}
\usepackage{inputenc}
\usepackage{graphicx,color}
\newtheorem{thm}{Theorem}[section]
\newtheorem{cor}[thm]{Corollary}

\newtheorem{prop}[thm]{Proposition}
\newtheorem{defn}[thm]{Definition}
\newtheorem{exam}[thm]{Example}
\newtheorem{rem}[thm]{Remark}
\numberwithin{equation}{section}

\begin{document}

\title [On some solvable Leibniz algebras and their completeness]{On some solvable Leibniz algebras and their completeness}

\author{K.K. Abdurasulov$^{1}$, B.A. Omirov$^{2}$, I.S. Rakhimov$^{1,3}$}

\address{$^1$ Institute of Mathematics Uzbekistan Academy of Sciences, Tashkent, Uzbekistan} \email{abdurasulov0505@mail.ru}

\address{$^2$ National University of Uzbekistan, Tashkent, Uzbekistan} \email{omirovb@mail.ru}

\address{$^3$ Faculty of Computer and Mathematical Sciences, Universiti Teknologi MARA (UiTM), Shah Alam, Malaysia $\&$ Institute of Mathematics Uzbekistan Academy of Sciences, Tashkent, Uzbekistan }
\email{isamiddin@uitm.edu.my}

\begin{abstract} The paper is devoted studying solvable Leibniz algebras with a nilradical possessing the codimension equals the number of its generators. We describe this class in non-split nilradical case. Then the case of split nilradical is worked out. We show that the results obtained earlier on this class of Leibniz algebras come as particular cases of the results of this paper. It is shown that such a solvable extension is unique. Finally, we prove that the solvable Leibniz algebras considered are complete.
\end{abstract}



\maketitle

\section{Introduction}

During the last 25 years the theory of Leibniz algebras has been a topic of active research. Some (co)homology and deformations properties; relations with $R$-matrices and Yang-Baxter equations; result on various types of decompositions; structure of semisimple, solvable and nilpotent Leibniz algebras; classifications of some classes of graded nilpotent Leibniz algebras were obtained in numerous papers devoted to Leibniz algebras. In fact, many results of theory of Lie algebras have been extended to the Leibniz algebras case. For instance, the classical results on Cartan subalgebras, Levi's decomposition, properties of solvable algebras with a given nilradical and others from the theory of Lie algebras are also true for Leibniz algebras. Due to the lack of non-antisymmetricity, however, matters will be not quite as simple as in the Lie algebras case. Nevertheless, many results from the case of Lie algebras are extended to Leibniz algebras.
In 2012 D. Barnes \cite{Barnes2} proved an analogue of Levi-Malcev Theorem\index{Levi-Malcev theorem} for Leibniz algebras. The author proved that a finite-dimensional Leibniz algebra is decomposed into a semidirect sum of its solvable radical and a semisimple Lie algebra. Therefore, the main problem of the description of finite-dimensional Leibniz algebras comes to the study of solvable Leibniz algebras. It was also pointed out there the non-uniqueness of the semisimple subalgebra appeared in the analogue of the Levi-Malcev theorem. Note that in the case of Lie algebras the semisimple part (factor Levi) is unique up to conjugation via an inner automorphism. However, the conjugation in the case of Leibniz algebras is no longer true, in general. Recently, in \cite{KuLO} the authors studied the conditions for the semisimple parts of the decomposition to be conjugated. Recall that A.I. Malcev proved that a solvable Lie algebra is uniquely determined by its nilradical \cite{Malcev} and some special kinds derivations of the nilradical. Mubarakzjanov proposed a method, which used the property that the dimension of the subspace complementary to the nilradical does not exceed the number of nil-independent derivations and as well as the number of generators of the nilradical \cite{mubor1}. Similar to that of the Lie algebras case, it is known that any non-nilpotent solvable Leibniz algebra can be constructed by means of its nilradical and nil-independent derivations of the nilradical \cite{Omirov}.

With a given Leibniz algebra we can associate a few Lie algebras, they are: the quotient algebra by the ideal generated by the squires (Lie core); the algebra of all right multiplication operators with the commutator operation; the quotient algebra by the ideal of right annihilators etc. Some of properties of the Leibniz algebras can be stated in terms of these algebras. For instance, the solvability of the Leibniz algebra is equivalent to the solvability of each of the Lie algebras mentioned. In \cite{Ernie} the concept of completeness of a Leibniz algebra was defined by means of its Lie core.

 Note that the square of a solvable Leibniz algebra is contained in its nilradical. Therefore, it is natural to use the nilradical for constructing a solvable algebra.
For a solvable finite-dimensional Leibniz algebra with a given nilradical the restriction of the right multiplication operator on elements of the subspace  complementary to the nilradical is non-nilpotent outer derivation of the nilradical. Therefore, the dimension of the subspace complementary to the nilradical is not greater than maximal number of nil-independent outer derivations of the nilradical of the solvable Leibniz algebra. This is one of the important reasons to study the derivations of nilpotent Leibniz algebras (especially, outer derivations or the first cohomology group).

In the case of Lie algebras to classify the solvable part a result of the paper \cite{mubor1} has been used (see \cite{AnCaGa}, \cite{CampoamorStursberg3}, \cite{NdogmoWinternitz}, \cite{SnoblWinternitz}, \cite{snoble} and \cite{WangDeng}). In the case of Leibniz algebras the analogue of Mubarakzjanov's \cite{mubor1} results has been also successfully implemented in \cite{4dim}, \cite{Omirov},   \cite{Filiform}, \cite{5dim} and \cite{Shab}.





The paper is organized as follows. In Preliminaries Section we give some definitions and background results needed in this study. Section 3 contains the description of solvable Leibniz algebras with maximal possible codimension of the nilradical. In Section 4 we prove the uniqueness of the solvable Leibniz extensions of the nilradical with the maximal possible codimension. Some comparisons of the results in this paper with those obtained earlier are given in Section 5. Finally, in Section 6 the completeness of the solvable Leibniz algebras studied is proved.

Throughout the paper we consider finite-dimensional vector spaces and algebras over the field of complex numbers. Moreover, in the table multiplications of algebras the products omitted are assumed to be zero.

\section{Preliminaries}
A Leibniz algebra\index{Leibniz algebra} over a field $\mathbb{F}$ is a vector space equipped with a $\mathbb{F}$-bilinear
map $[\cdot,\cdot]:L \times L\longrightarrow L$ satisfying the Leibniz identity
 \begin{equation} \label{Right} {\mathcal L}(x, y, z)=[[x,y],z]- [[x,z],y]-[x,[y,z]]=0,\end{equation} \index{Right Leibniz algebra identity}
for all $x, y, z\in L.$

 In fact, the definition above is a definition of the right Leibniz algebra, whereas the identity for the left Leibniz algebra is as follows
\begin{equation} \label{Left} [x,[z,y]]=[[x,z],y]+[z,[x,y]],\end{equation} \index{Left Leibniz algebra identity}
for all $x, y, z\in L.$
The passage from the right to the left Leibniz algebra can be easily done by considering a new product "$[\cdot,\cdot]_{opp}$" on the same vector space defined by "$[x,y]_{opp}=[y ,x]$". \index{opposite Leibniz algebra} Therefore, the results proved for the right Leibniz algebras $(L,[\cdot.\cdot])$ can be easily reformulated for the left Leibniz algebras $(L,[\cdot,\cdot]_{opp})$ and vice versa.

Note that the versions
\begin{equation} \label{Right1} [x,[y,z]]= [[x,y],z]-[[x,z],y], \ \mbox{for all}\  x,y,z\in L \end{equation}
and
\begin{equation} \label{Left} [[x,z],y]=[x,[z,y]]-[z,[x,y]], \ \mbox{for all}\  x,y,z\in L.\end{equation}
of the identities (\ref{Right}) and (\ref{Left}) are also often used.


The Leibniz identity implies that the elements $[x,x], [x,y]+[y,x]$ for any $x, y \in L$ are in the right annihilator $Ann_r(L)$ of $L$ and $Ann_r(L)$ is a two-sided ideal of $L$.


For a given Leibniz algebra $(L,[\cdot,\cdot ])$ the sequences of two-sided ideals defined recursively as follows:
\[L^1=L, \ L^{k+1}=[L^k,L],  \ k \geq 1, \qquad \qquad
L^{[1]}=L, \ L^{[s+1]}=[L^{[s]},L^{[s]}], \ s \geq 1,
\]
are said to be the \emph{lower central} and the \emph{derived series} of $L$, respectively.

\begin{defn} A Leibniz algebra $L$ is said to be
\emph{nilpotent} (respectively, \emph{solvable}), if there exists $n\in\mathbb N$ ($m\in\mathbb N$) such that $L^{n}=0$ (respectively, $L^{[m]}=0$).
\end{defn}

Let $R$ be a solvable Lie algebra with nilradical $N$ and $Q$ be the subspace  complementary
to the nilradical, i.e., $R=N\oplus Q$ as a sum of vector spaces.

\begin{defn} \label{defncomplete}
A Leibniz algebra $L$ is said to be complete, if its center is trivial and any derivation is inner.
\end{defn}

A complete Lie algebra is the Lie algebra of some algebraic group. So it is important in the study of algebraic groups, and its turn methods of algebraic groups is natural for the study of complete Lie algebras.

\begin{exam}
\
\begin{itemize}
  \item Obviously, simple and semisimple Lie algebras are complete.
  \item If a Lie algebra has trivial center and nondegenerate Killing form, then it is complete.
\item There is E.Schenkman's famous derivation
tower theorem \cite{Schenkman} asserting that an algebra with zero center can be subinvariantly embedded in an algebra with only inner derivations.
 \item In \cite{Qobil} the following result on completeness of solvable Lie algebras has been given which we make use.
  \begin{thm}\label{thmLieH1} Let $R=N\oplus Q$ be a solvable Lie algebra such that $dim Q=dim (N/N^2).$ Then $R$ is complete.
\end{thm}
\end{itemize}
\end{exam}



Let us review some results obtained earlier to motivate the research conducted in the paper. We start with the following classification result on naturally graded $p$-filiform non-Lie Leibniz algebras that has been obtained in \cite{Camacho}.

\begin{thm}\label{p-filiform} An arbitrary $n$-dimensional naturally graded non-split non-Lie $p$-filiform Leibniz algebra $(n-p\geq 4)$ is isomorphic to one of the following non-isomorphic algebras:

$p=2k$
$$\begin{array}{ll}
\mu_1:\left\{\begin{array}{ll}
[e_i,e_1]=e_{i+1}, & 1\leq i\leq n-2k-1,\\[1mm]
[e_1, f_j] =f_{k+j}, & 1\leq j\leq k,\\[1mm]
 \end{array}\right.&

\mu_2:\left\{\begin{array}{ll}
[e_i,e_1]=e_{i+1}, & 1\leq i\leq n-2k-1,\\[1mm]
[e_1,f_1]=e_{2}+f_{k+1}, &\\[1mm]
[e_i,f_1]=e_{i+1}, & 2\leq i\leq n-2k-1,\\[1mm]
[e_1, f_j] =f_{k+j}, & 2\leq j\leq k,\\[1mm]
 \end{array}\right.
 \end{array}$$

$p=2k+1$

$$\mu_3: \quad [e_i,e_1]=e_{i+1},\ 2\leq i\leq n-2k-1,\quad [e_2,f_i]=f_{k+i},\ 1\leq i\leq k.$$
where $\{e_1,e_2,\dots,e_{n-p},f_1,f_2,\dots,f_{p}\}$ is a basis
of the algebra.
\end{thm}


Let us denote by $R(N,m)$ the class of maximal (with respect to the dimension) solvable Leibniz algebras whose nilradical is $N$, that is, the class of solvable Leibniz algebras whose nilradical is $N$ and maximal possible codimension of the nilradical
is $m$.
 In \cite{p_filiform} it was proved that for the nilpotent Leibniz algebra $N=\mu_1$ the value of $m$ is equal to $k$ and the description of $R(\mu_1, k)$ was given as follows.

\begin{thm} \label{thmmu1}  An arbitrary algebra of the family $R(\mu_1, k)$ admits a basis such that the non-vanishing Leibniz brackets become:
\[R(\mu_1, k)(a_{i,j},\varphi_{i,j},\delta_{i,j}): \quad \left\{\begin{array}{ll}
[e_i,e_1]=e_{i+1}, & 1\leq i\leq n-2k-1,\\[1mm]
[e_1, f_j] =f_{k+j}, & 1\leq j\leq k,\\[1mm]
[e_i,x_j]=\sum\limits_{t=i+1}^{n-2k}a_{t-i+1,j}e_{t},&1\leq i\leq n-2k, \ 1\leq j\leq k,\\[1mm]
[f_i,x_i]=f_i,&1\leq i\leq k,\\[1mm]
[f_{k+i},x_i]=f_{k+i},&1\leq i\leq k,\\[1mm]
[x_i,f_i]=-f_i,&1\leq i\leq k,\\[1mm]
[x_i,f_j]=\varphi_{i,j}f_{k+j},&1\leq i\neq j\leq k,\\[1mm]
[x_i,x_j]=\delta_{i,j}e_{n-2k},&1\leq i, j\leq k.\\[1mm]
\end{array}\right.\]
\end{thm}

Note that the center of $R(\mu_1, k)(a_{i,j},\varphi_{i,j},\delta_{i,j})$ is trivial and the fact that all derivations are inner can be obtained by a straightforward computation. Therefore, $R(\mu_1, k)(a_{i,j},\varphi_{i,j},\delta_{i,j})$ is complete.

There the necessary and sufficient conditions of the existence of an isomorphism between two algebras of the family $R(\mu_1, k)(a_{i,j},\varphi_{i,j},\delta_{i,j})$ also was given as follows.

\begin{prop} \label{propmu1} Two algebras $R(\mu_1, k)^\prime(a_{i,j}^\prime,\varphi_{i,j}^\prime,\delta_{i,j}^\prime)$ and $R(\mu_1, k)(a_{i,j},\varphi_{i,j},\delta_{i,j})$ are isomorphic if and only if there exists $A\in \mathbb{C}^*$ such that
$$\begin{array}{llll}
a_{i,j}^\prime=\frac{a_{i,j}}{A^{i-1}}, & 2\leq i\leq n-2k+1, & 1\leq j\leq k,&\\
\varphi_{i,j}^\prime=\frac{\varphi_{i,j}}{A}, &1\leq i\neq
j\leq k,& \delta_{i,j}^\prime=\frac{\delta_{i,j}}{A^{n-2k}},
& 1\leq i,j\leq k.
\end{array}$$
\end{prop}

From Proposition \ref{propmu1} we can get no uniqueness of the solvable Leibniz algebras found in the case of $dimQ < dim(N/N^2)$ with non-split $N$ (while due to Proposition \ref{cor311} we have uniqueness in the case of $dimQ =dim(N/N^2)$. Indeed, for the nilpotent Leibniz algebra $\mu_1$ we have $dim(\mu_1/\mu_1^2)=k+1$ and $codim (\mu_1)=k$, but the class of solvable Leibniz algebras $R(\mu_1,k)^\prime(a_{i,j}^\prime,\varphi_{i,j}^\prime,\delta_{i,j}^\prime)$ contains infinite numbers of non-isomorphic algebra.

This is the essential difference between maximal solvable Lie and non-Lie Leibniz extensions for the case of $dimQ < dim(N/N^2)$ with non-split $N$. The uniqueness of complex maximal solvable Lie extensions of nilpotent algebras was proved in \cite{Qobil}.

\

\section{Description of solvable Leibniz algebras with maximal possible codimension of the nilradical}

\

In this section we describe complex finite-dimensional solvable Leibniz algebras $R$ with a given nilradical $N$. Write $R$ as the direct sum of vector spaces $R=N\oplus Q$, where $Q$ is the subspace complementary to $N$.

Here is an auxiliary result we make use.
\begin{prop} \label{thm31}
Let $R=N\oplus Q$ be such that the number of linear independent basis generators of $N$ equals the dimension of $Q$. Then $R$ admits a basis $\{e_1, e_2, \dots, e_n, x_1, x_2, \dots, x_k\}$ such that the following products

\begin{equation}\label{eq1}
\left\{\begin{array}{ll}
[e_i,x_i]=e_i,& 1\leq i\leq k,\\[1mm]
[x_i,e_i]=(b_i-1)e_i, &b_i\in\{0,1\}, 1\leq i\leq k,\\[1mm]
[e_i,x_j]=0,& 1\leq i\neq j\leq k,\\[1mm]
[x_i,x_j]=0,& 1\leq i,j\leq k,\\[1mm]
\end{array}\right.
\end{equation}
hold true.

\end{prop}
\begin{proof}
Let $\{x_1, x_2, \dots, x_k\}$ be a basis of $Q$, $E=\{e_1, e_2, \dots, e_s\}$ be the set of linear independent generators of $N$. Denote by $\tau$ the nilindex of $N$, i.e., $N^{\tau-1}\neq0$ and $N^{\tau}=0$.

Introduce the following notations. Let $M$ be the span of all the right-normed words the length $m$ of the alphabet $E$. We break up $M$ into the following three subspaces $M_0, M_1$ and $M_2$, where $M_0$ is spanned by the right-normed words without $e_1$, $M_1$ is spanned by the right-normed words with only one $e_1$ and $M_2$ is the subspace spanned by the right-normed words with more than one involvement of $e_1$, respectively.

Obviously,
$$M_0=\bigcup_{t=2}^{k}M_{0,t},\ \mbox{and}\ M_1=\bigcup_{t=2}^{k}M_{1,t},$$
where $M_{0,t}$ and $M_{1,t}$ are the subspaces of $M_0$ and $M_1$, respectively,  spanned by the words with the minimal subindex of $\{e_2, e_3, \dots, e_s\}$ equals $t$.

Let set
$$\{u_{t,1}, u_{t,2}, \dots, u_{t,s_t}\}, \quad \{v_{t,1}, v_{t,2}, \dots, v_{t,k_t}\}, \ \mbox{and}\ \{w_1, w_2, \dots, w_p\}$$
to be bases of $M_{0,t}, M_{1,t}$ and $M_{2},$ respectively.

Leibniz identity gives that any word can be written as a linear combination of the right-normed words of the same length. Therefore, without loss of generality we can assume that the basis elements $u_{t,i}, v_{t,j}, w_{l}$ to be chosen as the right-normed words of the alphabet $E$.

Note also that for any $m$ one has
 $R/N^{m}=(N\oplus Q)/N^{m}=N/N^m\oplus Q$ and $N/N^m$ is the nilradical of the quotient algebra $R/N^{m}$ that gives chance to use the method of mathematical induction by $m$ ($2\leq m \leq \tau+1$) taking into account that the quotient algebra $R/N^{m}$ admits a basis such that the counterpart of the products \eqref{eq1} hold true.

Let $m=2$. Since $N/N^2$ is abelian, the solvable quotient Leibniz algebra algebra $R/N^2$ has abelian nilradical. Due to a result of the paper \cite{Abelian} the algebra $R/N^2$ has to have the table of multiplications \eqref{eq1}.

Let us assume that the products (\ref{eq1}) to be true for $m$. In order to prove (\ref{eq1}) for $(m+1)$ in the quotient algebra $R/N^{m+1}$ we need to consider the products:
$$[e_i,x_i],\ [x_i,e_i],\ [e_i,x_j],\ [x_j,e_i],\ [x_i,x_j],\ \ \ 1\leq i\neq j \leq k.$$

For the sake of a convenience further we use the congruences modulo $N^{m+1}$ without writing $N^{m+1}$.
From induction supposition we have

%

\qquad\qquad $[e_1,x_1]\equiv e_1+\sum\limits_{j=2}^{k}\sum\limits_{l=1}^{s_{j}}\alpha_{j,l}^1u_{j,l}+
\sum\limits_{j=2}^{k}\sum\limits_{l=1}^{k_{j}}\beta_{j,l}^1v_{j,l}
+\sum\limits_{q=1}^{p}\gamma_{q}^1w_{q},$

\qquad\qquad\ $[e_1,x_i]\equiv \sum\limits_{j=2}^{k}\sum\limits_{l=1}^{s_{j}}\alpha_{j,l}^iu_{j,l}+
\sum\limits_{j=2}^{k}\sum\limits_{l=1}^{k_{j}}\beta_{j,l}^iv_{j,l}
+\sum\limits_{q=1}^{p}\gamma_{q}^iw_{q},$ where $2\leq i \leq k.$

Observe that $[w_q,x_1]\equiv l_{q}w_q$ with $l_{q}\geq2.$ Indeed, applying the Leibniz identity $[[x,y],z]=[x,[y,z]]+[[x,z],y]$ and the induction supposition we derive
$$[w_q,x_1]=[[[...[e_{i_1},e_{i_2}],e_{i_{3}}],...],e_{i_m}],x_1]\equiv l_{q}w_q$$
where $l_{q}=\alpha_{i_{m}}+\alpha_{i_{m-1}}+...+\alpha_{i_{2}}+\alpha_{i_{1}}$ with $\alpha_{i_{j}}=\begin{cases} 1,\ \mbox{if} \ i_{j}=1,\\[1mm]
0,\ \mbox{if} \ i_{j}\neq 1.\\[1mm]
\end{cases}$

Since the involvement of $e_1$ in $w_{q}$ is more than once we conclude that $l_{q}\geq2$. Similarly, one can derive $[v_{i,l},x_i]\equiv r_{i,l}v_{i,l},$ where $r_{i,l}\geq 1.$

Apply the base change
$$e_1'=e_1+\sum\limits_{j=2}^{k}\sum\limits_{l=1}^{s_{j}}\alpha_{j,l}u_{j,l}-
\sum\limits_{j=2}^{k}\sum\limits_{l=1}^{k_{j}}\frac{\beta_{j,l}^{j}}{r_{j,l}}v_{j,l}
+\sum\limits_{q=1}^{p}\frac{\gamma_{q}}{1-l_{q}}w_{q},$$
to get

\qquad\qquad\qquad\qquad$[e_1,x_1]\equiv e_1+\sum\limits_{j=2}^{k}\sum\limits_{l=1}^{k_{j}}\beta_{j,l}v_{j,l}$\\
and

\qquad\qquad\qquad\qquad$[e_1,x_i]\equiv \sum\limits_{j=2}^{k}\sum\limits_{l=1}^{s_{j}}{\alpha_{j,l}^i}u_{j,l}+\sum\limits_{\tiny{\begin{array}{c} j=2,\\ j\neq i\end{array}}}^{k}\sum\limits_{l=1}^{k_{j}}{\beta_{j,l}^i}v_{j,l}+\sum\limits_{q=1}^{p}{\gamma_{q}^i}w_{q}, 2\leq i \leq k.$

Leibniz identities ${\mathcal L}(e_1, x_1, x_j)={\mathcal L}(e_1, x_i, x_q)=0$ imply

\qquad\qquad\qquad\qquad\qquad $[e_1,x_1]\equiv e_1,\ \ [e_1,x_i]\equiv \sum\limits_{\tiny{\begin{array}{c} j=2,\\ j\neq i\end{array}}}^{k}\sum\limits_{l=1}^{k_{j}}\beta_{j,l}^iv_{j,l},$

\begin{equation}\label{eq2}
[e_1,x_2]\equiv0, \quad [e_1,x_i]\equiv \sum\limits_{j=2}^{i-1}\sum\limits_{l=1}^{k_{j}}\beta_{j,l}^iv_{j,l},\ 3\leq i\leq k.
\end{equation}

Now we prove that
\begin{equation}\label{eq3}
[e_1,x_{h}]\equiv 0, \quad [e_1,x_i]\equiv \sum\limits_{j=h}^{i-1}\sum\limits_{l=1}^{k_{j}}\beta_{j,l}^iv_{j,l},\
\mbox{where}\ 2\leq h \leq i-1\ \mbox{and}\ 3\leq i\leq k,
\end{equation}
the \eqref{eq2} being the base of the induction. Let us assume that \eqref{eq3} holds true.

Thus, we obtain $$[e_1,x_1]\equiv e_1,\ \ [e_1,x_i]\equiv 0.$$

We put
$$[x_1,e_1]\equiv (b_1-1)e_1+\sum\limits_{j=2}^{k}\sum\limits_{l=1}^{s_{j}}a_{j,l}u_{j,l}+
\sum\limits_{j=2}^{k}\sum\limits_{l=1}^{k_{j}}b_{j,l}v_{j,l}+\sum\limits_{q=1}^{p}c_{q}w_{q}$$

$$[x_i,e_1]\equiv \sum\limits_{j=2}^{k}\sum\limits_{l=1}^{s_{j}}m_{j,l}^iu_{j,l}+\sum\limits_{j=2}^{k}\sum\limits_{l=1}^{k_{j}}n_{j,l}^iv_{j,l}+
\sum\limits_{t=1}^{p}h_{t}^iw_{t},\  \ \ 2\leq i\leq k.$$

From ${\mathcal L}(x_1, e_1, x_i)={\mathcal L}(x_i, e_1, x_1)={\mathcal L}(x_i, e_1, x_t)=0,\ \mbox{where}\ 1\leq i\leq k,\ \ 2\leq t\leq k,$ we get
$$[x_1,e_1]\equiv (b_1-1)e_1, \quad [x_i,e_1]\equiv 0, \ \mbox{for}\ 2\leq i\leq k.$$

These yield the products
$$[e_1,x_1]\equiv e_1, \quad [e_1,x_i]\equiv [x_i,e_1]\equiv 0, \quad [x_1,e_1]\equiv (b_1-1)e_1, \ \mbox{where}\ b_1\in \{1,0\}, \ \mbox{and}\ 2\leq i\leq k.$$

Successively applying the arguments used above for $i$ to the products $[e_i,x_j],\ [x_j,e_i]$ and taking into account that the base changes do not affect to the products $[e_t,x_j]$ and $[x_j,e_t]$ for $1\leq t \leq i-1,$ we deduce
$$[e_i,x_i]\equiv e_i, \quad [e_i,x_j]\equiv [x_j,e_i]\equiv 0, \quad [x_i,e_i]\equiv (b_i-1)e_i,\mbox{where}\ b_i\in \{1,0\}, \ \mbox{and}\ 2\leq i\neq j \leq k.$$

To prove $[x_i,x_j]\equiv 0$ for $1\leq i, j \leq k$ we need to consider the subspaces $M_t, \ 1\leq t\leq k,$
where $M_t$ is a subspace of $M$ spanned by the words with minimal subindex of $\{e_1, e_2, \dots, e_k\}$ equals $t.$

Again successively applying the arguments used above for the values of $i, \ \ 2\leq i \leq k$ to the products
$[x_i,x_j],\ [x_j,x_i],\ [x_i,x_i]$ and taking into account that
$[x_t,x_j]\equiv[x_j,x_t]\equiv 0$ for $1\leq t \leq i-1$ and the fact that in the expression of $[x_i,x_j]$ participate
only the basis elements $v_{t,q}$ with $i\leq t \leq k, \ 1\leq q \leq s_t$, we get
$[x_i,x_j]=0,\ \ 1\leq i,j\leq k.$

Taking in the products (\ref{eq1}) $m=\tau+1$ we complete the proof of the theorem.
\end{proof}

We set $[e_i,e_j]=\sum\limits_{t=k+1}^{n}\gamma_{i,j}^te_t, \ 1\leq i, j\leq n.$

\begin{thm} \label{thm35} Let $R=N\oplus Q$ be a solvable Leibniz algebra such that $dimQ=dim (N/N^2)=k.$ Then $R$ admits a
basis $\{e_1, e_2, \dots, e_n, x_1, x_2, \dots, x_k\}$ such that the table of multiplications of $R$ is given as follows
\begin{equation}\label{table1}
\left\{\begin{array}{ll}
[e_i,e_j]=\sum\limits_{t=k+1}^{n}\gamma_{i,j}^te_t,& 1\leq i, j\leq n,\\[1mm]
[e_i,x_i]=e_i,& 1\leq i\leq k,\\[1mm]
[x_i,e_i]=(b_i-1)e_i, & b_i \in \{0,1\}, 1\leq i\leq k,\\[1mm]
[e_i,x_j]=\alpha_{i,j}e_i,& k+1\leq i\leq n,\ \ 1\leq j\leq k,\\[1mm]
[x_j,e_i]=\sum\limits_{t=1}^{q}\beta_{j,i}^{i_t}e_{i_t}, &  k+1\leq i\leq n, \ 1\leq j\leq k,\\[1mm]
\end{array}\right.
\end{equation}
where $\alpha_{i,j}$ is the number of the involvement of generator basis element $e_j$ in non generator basis element $e_i$.
\end{thm}
\begin{proof} From Proposition \ref{thm31} we have the existence of a basis $\{e_1, e_2, \dots, e_n, x_1, x_2, \dots, x_k\}$ with the products (\ref{eq1}). Note that in Proposition \ref{thm31} we dealt with the generator basis elements of $N$. Remind that non generator basis elements of $N$ are given as the right normed words of the alphabet $E$.

Similarly, like in Proposition \ref{thm31} we have
$$[e_i,x_j]=[[[\cdots [e_{i_1},e_{i_2}],e_{i_{3}}], \dots],e_{i_t}],x_j]=\alpha_{i,j}e_i, \ k+1\leq i\leq n, 2\leq t \leq \tau,$$
where $\alpha_{i,j}=\alpha_{i_{1},j}+\alpha_{i_{2},j}+...+\alpha_{i_{t-1},j}+\alpha_{i_{t},j}$ with $\alpha_{i_{l},j}=\begin{cases}
1,\ \mbox{if} \ i_{l}=j,\\[1mm]
0,\ \mbox{if} \ i_{l}\neq j,\\[1mm]
\end{cases}$ for $1\leq l\leq t.$

Therefore, $[e_i,x_j]=\alpha_{i,j}e_i, \  k+1\leq i\leq n, \ 1\leq j\leq k.$

We set
$$[x_j,e_i]=\sum\limits_{t=k+1}^{n}\beta_{j,i}^te_t, \ k+1\leq i\leq n, \ 1\leq j\leq k.$$

Consider
$$[x_j,[e_i,x_p]]=[[x_j,e_i],x_p]-[[x_j,x_p],e_i]=\left[\sum\limits_{t=k+1}^{n}\beta_{j,i}^te_t,x_p\right]=
\sum\limits_{t=k+1}^{n}\alpha_{t,p}\beta_{j,i}^te_t.$$
On the other hand we have
$$[x_j,[e_i,x_p]]=\alpha_{i,p}[x_j,e_i]=\alpha_{i,p}\sum\limits_{t=k+1}^{n}\beta_{j,i}^te_t.$$
Therefore,
\begin{equation}\label{eq10}
\beta_{j,i}^t(\alpha_{t,p}-\alpha_{i,p})=0,\ k+1\leq i,t\leq n,\ 1\leq j,p\leq k.
\end{equation}

Let $\{e_{i_1}, e_{i_2}, \dots, e_{i_q}\}$ be a set of all non generator basis elements which have the same
structure as $e_i, \ \ k+1\leq i \leq n$. Then from (\ref{eq10}) we conclude that
 $$\beta_{j,i}^t=0,\ \mbox{for} \ k+1\leq t\leq n,\ \mbox{for}\ 1\leq j\leq k,\ t\neq i_1,i_2,\ldots,i_q.$$
Hence,
$$[x_j,e_i]=\sum\limits_{t=1}^{q}\beta_{j,i}^{i_t}e_{i_t}, \ \mbox{where} \ k+1\leq i\leq n,\ \ 1\leq j\leq k.$$
\end{proof}

\begin{rem} \label{prop38}\emph{}

\begin{itemize}
  \item Note that in Theorem \ref{thm35} the subspace $Q=Span_{\mathbb{C}}\{x_1, \dots, x_k\}$ forms an abelian subalgebra of $R$. Therefore, we get a right action of $Q$ on the nilradical $N$, which decomposes $N$ into one-dimensional root subspaces. In fact, similarly to Lemma 3.2 in \cite{Qobil} we can prove that this is nothing but a torus of the right action on $N$ (recall that the torus is an abelian subalgebra of the Lie algebra $Der(N)$ consisting of diagonalizable derivations). Since in non-Lie Leibniz algebras the anti-commutativity property is longer true the action of $Q$ on $N$ gives another root decomposition, which need not be diagonal action (the root subspaces with respect to the left action of $Q$ on $N$ happen to be non one-dimensional);
  \item If among the non generator basis elements of the nilradical $N$ there are no elements of the same structure, then $$[x_j,e_i]=\beta_{j,i}e_{i}, \quad k+1\leq i\leq n, \ 1\leq j\leq k;$$
  \item From (\ref{table1}) we conclude that $\alpha_{i,j}$ are uniquely determined by the structure of the nilradical $N;$
  \item Note that for any $i,\ 1\leq i\leq n$ there exists $j,\ 1\leq j\leq k$ such that $\alpha_{i,j} \neq 0.$
\end{itemize}

%
%
%
\end{rem}

Below we give the assertion of the theorem above for Lie algebras case. This case has been studied in  \cite{liecase}, where the solvable Lie algebras considered had the nilradical with the codimension equals to the number of generators of the nilradical.

\begin{cor} \label{cor36} Let $R=N\oplus Q$ be a solvable Lie algebra such that $dimQ=dim (N/N^2)=k.$ Then $R$ admits a basis $\{e_1, e_2, \dots, e_n, x_1, x_2, \dots, x_k\}$ such that the table of multiplications of $R$ has the following unique form:
\begin{equation}\label{eq15} \left\{\begin{array}{ll}
[e_i,e_j]=\sum\limits_{t=k+1}^{n}\gamma_{i,j}^te_t,& 1\leq i\neq j\leq n,\\[1mm]
[e_i,x_i]=-[x_i,e_i]=e_i,& 1\leq i\leq k,\\[1mm]
[e_i,x_j]=-[x_j,e_i]=\alpha_{i,j}e_i,& k+1\leq i\leq n,\ \ 1\leq j\leq k.\\[1mm]
\end{array}\right.\end{equation}
\end{cor}

Below we give a result which states that if the nilradical is non split Lie algebra under the condition of the theorem on the
 codimension of the nilradical then the Leibniz algebras described in the theorem become Lie algebras.

\begin{thm} \label{lem39} Let $R=N\oplus Q$ be a solvable Leibniz algebra such that $dimQ=dim (N/N^2)=k$ and $N$
be a non split Lie algebra. Then $R$ is a Lie algebra with the table of multiplications $(\ref{eq15})$.
\end{thm}

\begin{proof} Due to Theorem \ref{thm35} we have the existence of a basis $\{e_1, e_2, \dots, e_n, x_1, x_2, \dots, x_k\}$
such that the table of multiplications of $R$ has the form (\ref{table1}). Since the nilradical of $R$ is non split, the generator
 basis elements of the nilradical $N$ are not in the right annihilator, i.e., $e_1, e_2, \dots, e_k \not \in Ann_r(R)$
 (otherwise the nilradical $N$ would be split).

From $b_ie_i=[e_i,x_i]+[x_i,e_i]\in Ann_r(R)$ we get $b_i=0,\ 1 \leq i\leq k.$

Similarly as in the proof of Theorem \ref{thm35} we can assume that non generator basis elements $e_i,\ k+1\leq i\leq n$ have the form:
$e_i=[[\cdots [[e_{i_1},e_{i_2}],e_{i_{3}}],\cdots],e_{i_m}],\ \mbox{where}\ 2\leq m\leq \tau$.

Therefore,

 $[x_j,e_i]=[x_j,[[\cdots [[e_{i_1},e_{i_2}],e_{i_{3}}],\cdots,e_{i_{m-1}}],e_{i_m}]]$

\ $$=[[x_j,[[\cdots [[e_{i_1},e_{i_2}],e_{i_{3}}],\cdots ],e_{i_{m-1}}]],e_{i_m}]-[[x_j,e_{i_m}],[[\cdots [[e_{i_1},e_{i_2}],e_{i_{3}}],\cdots ],e_{i_{m-1}}]]$$
$$=[[x_j,[[\cdots [[e_{i_1},e_{i_2}],e_{i_{3}}],\cdots ],e_{i_{m-1}}]],e_{i_m}]-\delta_{j,i_m}[e_{i_m},[[\cdots [[e_{i_1},e_{i_2}],e_{i_{3}}],\cdots ],e_{i_{m-1}}]]$$
$$=[[x_j,[[\cdots [[e_{i_1},e_{i_2}],e_{i_{3}}],\cdots ],e_{i_{m-1}}]],e_{i_m}]+\delta_{j,i_m}[[\cdots [[e_{i_1},e_{i_2}],e_{i_{3}}],\cdots ],e_{i_{m-1}}],e_{i_m}]$$
$$ \cdots $$
\qquad\quad \ \ \ $=(\delta_{j,i_{m}}+\delta_{j,i_{m-1}}+\dots+\delta_{j,i_{2}}+\delta_{j,i_{1}})[[\cdots [[e_{i_1},e_{i_2}],e_{i_{3}}],\cdots,e_{i_{m-1}}],e_{i_m}]$

\quad\quad \ \ $=(\delta_{j,i_{m}}+\delta_{j,i_{m-1}}+\dots+\delta_{j,i_{2}}+\delta_{j,i_{1}})e_{i},$

where
$$\delta_{j,i_{l}}=\begin{cases}
-1,\ \mbox{if} \ i_{l}=j,\\[1mm]
0,\ \mbox{if} \ i_{l}\neq j,\\[1mm]
\end{cases} 1\leq l\leq m.$$

Consequently, $$\delta_{j,i_{m}}+\delta_{j,i_{m-1}}+ \dots+\delta_{j,i_{2}}+\delta_{j,i_{1}}=-\alpha_{i,j},\ \ \ 1\leq j\leq k,\ k+1\leq i\leq n,\ 2\leq m\leq  \tau. $$

Thus, we obtain $[x_j,e_i]=-\alpha_{i,j}e_i,\ \ \ k+1\leq i\leq n,\ 1\leq j\leq k$.
\end{proof}

\section{The uniqueness of solvable Leibniz extensions of the nilradical with the maximal possible codimension
}

\

In this section we treat the uniqueness of the solvable Leibniz algebras of the form $R=N\oplus Q$ with $dimQ=dim (N/N^2)=k$. Here we keep using the basis of Theorem \ref{thm35}.
Remind that $\{e_1, e_2, \dots, e_{k_1}\}\cap Ann_r(N)=\emptyset$ and $\{e_{k_1+1}, e_{k_1+2}, \dots, e_{k}\}\subset Ann_r(N).$

\begin{prop} \label{cor311} Let $R=N\oplus Q$ be a solvable Leibniz algebra such that $dimQ=dim (N/N^2)=k$ and $N$ be
 a non split nilradical. Then $R$ is isomorphic to a unique algebra with the following table of multiplications
\begin{equation}\label{eq22}\left\{\begin{array}{ll}
[e_i,e_j]=\sum\limits_{t=k+1}^{n}\gamma_{i,j}^te_t,& 1\leq i, j\leq n,\\[1mm]
[e_i,x_i]=e_i,& 1\leq i\leq k,\\[1mm]
[x_i,e_i]=-e_i, & 1\leq i\leq k_1,\\[1mm]
[e_i,x_j]=\alpha_{i,j}e_i,& k+1\leq i\leq n,\ \ 1\leq j\leq k,\\[1mm]
[x_j,e_i]=\sum\limits_{t=1}^{q}\beta_{j,i}^{i_t}e_{i_t}, &  k+1\leq i\leq n, \ 1\leq j\leq k,\\[1mm]
\end{array}\right.\end{equation}
where $\alpha_{i,j}$ is the number of involvement of the generator basis
element $e_j$ in the writing of the non generator basis element
$e_i$.
\end{prop}

\begin{proof} From Theorem \ref{thm35} we have the table of multiplications (\ref{table1}). It is easy to see that the structure constants $\alpha_{i,j}$ are uniquely determined by the table of multiplications of the nilradical $N$.

For $1\leq i\leq k_1$ we have $b_ie_i=[e_i,x_i]+[x_i,e_i]\in Ann_r(R)$, which imply $b_i=0$ and $[x_i, e_i]=-e_i,$ for $1 \leq i\leq k_1$.

Observe that for $i$ and $j$ such that $k_1+1 \leq i \leq k, 1 \leq j \leq k_1$ we also have $[e_i,e_j]=[e_i,e_j]+[e_j,e_i] \in Ann_r(R)$. Therefore,
$$0=[x_i,[e_i,e_j]]=[[x_i,e_i],e_j]-[[x_i,e_j],e_i]=(b_i-1)[e_i,e_j].$$
Since the nilradical $N$ is non split, this yields that for any $e_{i}$, where $k_1+1 \leq i \leq k$ there exists
$e_j, \ 1 \leq j \leq k_1$ such that $[e_i,e_j]\neq 0$. Hence, $b_i=1$ and $[x_i,e_i]=0, \ k_1+1\leq i\leq k$.

Let us take a non generator basis element $e_i$, where $k+1\leq i\leq n$ in the form
$e_i=[[\cdots [[e_{i_1},e_{i_2}],e_{i_{3}}],\cdots],e_{i_m}],\ 2\leq m\leq \tau$. Then the uniqueness of the structure constants $\beta_{j,i}^{i_t}$ in the products $[x_j,e_i]$ for $k+1\leq i\leq n, \ 1\leq j\leq k$ follows from the uniqueness of $\gamma_{i,j}^t$, the induction by $m$ and the equality

$[x_j,e_i]=[x_j,[[\cdots [[e_{i_1},e_{i_2}],e_{i_{3}}],\cdots,e_{i_{m-1}}],e_{i_m}]]$

\qquad \ \ \ $=[[x_j,[[\cdots [[e_{i_1},e_{i_2}],e_{i_{3}}],\cdots ],e_{i_{m-1}}]],e_{i_m}]-[[x_j,e_{i_m}],[[\cdots [[e_{i_1},e_{i_2}],e_{i_{3}}],\cdots ],e_{i_{m-1}}]].$
\end{proof}

Let us now consider more general case, where the nilradical of the solvable Leibniz algebra is split and without abelian ideals.
%

\begin{thm}\label{thm33} Let $R=N\oplus Q$ be a solvable Leibniz algebra with $dimQ=dim (N/N^2)=k$ and $N=\bigoplus\limits_{t=1}^{s}N_t$, where $N_t$ is non split and non abelian ideal of $N$ with $dimN_t=n_t$.
Then
\begin{itemize}
  \item the algebra $R$ admits a basis $\{e_1^t, e_2^t, \dots, e_{k_t^1}^t, e_{k_t^1+1}^t, e_{k_1^1+2}^t, \dots, e_{k_t}^t, x_1^t, \dots,x_{k_t}^t\},\ 1\leq t\leq s,$
such that  the table of multiplications of $R$ has the following form:
\begin{equation}\label{table22}\left\{\begin{array}{ll}
[e_i^t,e_j^t]=\sum\limits_{p=k_t+1}^{n_1}\gamma_{i,j}^{t,p}e_t,& 1\leq i, j\leq n_t,\ 1\leq t\leq s,\\[1mm]
[e_i^t,x_i^t]=e_i^t,& 1\leq i\leq k_t,\ 1\leq t\leq s,\\[1mm]
[x_i^t,e_i^t]=-e_i^t,&1\leq i\leq k_t^1,\ 1\leq t\leq s,\\[1mm]
[e_i^t,x_j^t]=\alpha_{i,j}^te_i^t,& k_t\leq i\leq n_t,\  1\leq j\leq k_t, \ 1\leq t\leq s,\\[1mm]
[x_j^t,e_i^t]=\sum\limits_{l=1}^{q}\beta_{t,j,i}^{i_l}e_{i_l}^t, &  k_t+1\leq i\leq n_t, \ 1\leq j\leq k_t,\ 1\leq t\leq s;\\[1mm]
\end{array}\right.\end{equation}

 \item  $R$ is a unique, up to isomorphism, algebra and it has the form $R=\bigoplus\limits_{t=1}^{s}R_t,$ with $R_t=N_t\oplus Q_t,$ where $N_t$ is the nilradical of $R_t$ and $Q_t$ is the subspace complementary to $N_t$.
\end{itemize}
\end{thm}
\begin{proof} Consider the basis $\{e_1^t, e_2^t, \dots, e_{k_t}^t,e_{k_t+1}^t,\dots, e_{n_t}^t\}$ of the ideal $N_t,$ where $\{e_1^t, e_2^t, \dots, e_{k_t}^t\}$ are generators of $N_t$, $1\leq t \leq s$ and such that $\{e_1^t, e_2^t, \dots, e_{k_t^1}^t\}\cap Ann_r(N_t)=\emptyset,$ $\{e_{k_t^1+1}^t, e_{k_1^1+2}^t, \dots, e_{k_t}^t\}\subset Ann_r(N_t).$ From the conditions of the theorem and Theorem \ref{thm35} we get the existence of the basis
$\{e_1^t, e_2^t, \dots, e_{k_t}^t,e_{k_t+1}^t,\dots, e_{n_t}^t,x_1^t, \dots,x_{k_t}^t\}$, where
$\{e_1^t, e_2^t, \dots, e_{k_t}^t,e_{k_t+1}^t,\dots, e_{n_t}^t\}$ is a basis of $N_t,$ and
$\{e_1^t, e_2^t, \dots, e_{k_t}^t\}$ are generators of $N_t$, $1\leq t \leq s.$ Obviously, in
this basis $R$ has the following products
\begin{equation}\left\{\begin{array}{ll}\label{3.16}
[e_i^t,x_i^t]=e_i^t,& 1\leq i\leq k_t,\  1\leq t\leq s,\\[1mm]
[x_i^t,e_i^t]=(b_i^t-1)e_i^t,& b_i^t \in \{0,1\},\  1\leq i\leq k_t,\ 1\leq t\leq s,\\[1mm]
[e_i^t,x_j^t]=\alpha_{i,j}^te_i^t,& k_t+1\leq i\leq n_t,\  1\leq j\leq k_t,\\[1mm]
[x_i^{t_1},x_j^{t_2}]=0,&1\leq i\leq k_{t_1},\ 1\leq j\leq k_{t_2},\  1\leq t_1,t_2\leq s.\\[1mm]
\end{array}\right.\end{equation}

Since $[N_{t_1},N_{t_2}]=0$ for $t_1\neq t_2$, then in $e_i^{t_1}$ with $k_{t_1}+1\leq i\leq n_{t_1}$ the elements $e_j^{t_2},$ $1\leq j\leq k_{t_2}$ are not involved. Therefore, $[e_i^{t_1},x_j^{t_2}]=[x_j^{t_2},e_i^{t_1}]=0$.

Using the products (\ref{3.16}) arguing similarly to that in the proof of Theorem \ref{thm35} we obtain
$$[x_j^t,e_i^t]=\sum\limits_{p=1}^{q}\beta_{j,i}^{i_p,t}e_{i_p}^t \ , \quad  k_t+1\leq i\leq n_t, \ 1\leq j\leq k_t,\ 1\leq t\leq s.$$

Thus, for each $t$, $1\leq t\leq s$ the subspace $R_t$ spanned by the basis elements $\{e_1^t, e_2^t,
\dots, e_{k_t}^t,e_{k_t+1}^t,\dots, e_{n_t}^t,x_1^t,\dots,x_{k_t}^t\}$ forms an ideal of $R$. Moreover,
$R=\bigoplus\limits_{t=1}^{s}R_t$ and $R_t$ has the form $R_t=N_t\oplus Q_t$, where $N_t$ is the nilradical of
$R_t$ and $Q_t=Span\{x_1^t,\dots,x_{k_t}^t\}$ is the subspace complementary to $N_t$, which is an abelian
subalgebra of $R_t$. The uniqueness of $R$ and (\ref{table22}) is imediate from Proposition \ref{cor311}.
\end{proof}

For the case of solvable Leibniz algebras whose nilradical is Lie algebra without abelian ideals we have the following result.

\begin{cor}\label{lie}
Let $R=N\oplus Q$ be a solvable Leibniz algebra with $dimQ=dim (N/N^2)=k$ and $N=\bigoplus\limits_{t=1}^{s}N_t$, where $N_t$ is non
split non abelian Lie ideal of $N$. Then

\begin{itemize}
  \item algebra $R$ admits a basis $\{e_1^t,\dots,e_{n_t}^t,x_1^t,\dots,x_{k_t}^t\},\ 1\leq t\leq s,$
such that  the table of multiplications of $R$ has the following form:
\begin{equation}\label{table2}\left\{\begin{array}{ll}
[e_i^t,e_j^t]=\sum\limits_{p=k_t+1}^{n_1}\gamma_{i,j}^{t,p}e_t,& 1\leq i, j\leq n_t,\ 1\leq t\leq s,\\[1mm]
[e_i^t,x_i^t]=-[x_i^t,e_i^t]=e_i^t,& 1\leq i\leq k_t,\ 1\leq t\leq s,\\[1mm]
[e_i^t,x_j^t]=-[x_j^t,e_i^t]=\alpha_{i,j}^te_i^t,& k_t\leq i\leq n_t,\  1\leq j\leq k_t, \ 1\leq t\leq s;\\[1mm]
\end{array}\right.\end{equation}

 \item $R$ is a unique, up to isomorphism, Lie algebra given as follows: $R=\bigoplus\limits_{t=1}^{s}R_t,$ with $R_t=N_t\oplus Q_t,$ where $N_t$ is the nilradical of $R_t$ and $Q_t$ is the subspace complementary to $N_t$.
\end{itemize}
\end{cor}

\begin{rem} Note that the condition that the underlying field is $\mathbb{C}$ in Proposition \ref{cor311}, Theorem \ref{thm33} and Corollary \ref{lie} is essential. Indeed, if we consider the three-dimensional Heisenberg Lie algebra
\begin{equation}
\label{exm1} H_1: \quad [e_2,e_3]=-[e_3, e_2]=e_1,
\end{equation}
as the nilradical then there exists, up to isomorphism, a unique complex solvable five-dimensional Lie algebra with the nilradical $H_1$ having the table of multiplications as follows
\begin{equation}\label{exm2}
[e_2,e_3]=e_1,\ [e_2,x_1]=e_2,\ [e_1,x_1]=e_1,\ [e_3,x_2]=e_3,\ [e_1,x_2]=e_1.
\end{equation}

However, if we take $\mathbb{R}$ as the underlying field, then according to a result of the paper \cite{mubor} there are two non isomorphic solvable Lie algebras

$\emph{g}_{5,36}: \ [e_2,e_3]=e_1,\ [e_1,e_4]=e_1,\ \ [e_2,e_4]=e_2,\ [e_2,e_5]=-e_2,\ [e_3,e_5]=e_3,$\\
and

$\emph{g}_{5,37}: \ [e_2,e_3]=e_1,\ [e_1,e_4]=2e_1,\ [e_2,e_4]=e_2,\ [e_3,e_4]=e_3,\ \ \ [e_2,e_5]=-e_3,\ [e_3,e_5]=e_2$\\
one of them, namely $\emph{g}_{5,36}$, being isomorphic to \eqref{exm2} via the
base change $x_1=e_4,\ x_2=e_4+e_5.$


\end{rem}

Finally, below we give a result on solvable Leibniz algebras with the nilradical in more general form $N=\bigoplus\limits_{t=1}^{s}N_t\oplus\mathbb{C}^p$.

\begin{thm} \label{thm3.13} Let $R=N\oplus Q$ be a solvable Leibniz algebra with $dimQ=dim (
N/N^2)=k+p$ and $N=\bigoplus\limits_{t=1}^{s}N_t\oplus\mathbb{C}^p$, where $N_t$ is non
split non abelian ideal of $N$. Then

\begin{itemize}
  \item the algebra $R$ admits a basis $\{e_1^t, e_2^t, \dots, e_{k_t^1}^t, e_{k_t^1+1}^t, e_{k_1^1+2}^t, \dots, e_{k_t}^t, x_1^t, \dots,x_{k_t}^t, f_1,\dots,f_p,y_1,\dots,y_p\},\ 1\leq t\leq s,$
such that  the table of multiplications of $R$ has the following form:
\begin{equation}R_1(b_1,b_2,\dots.b_p):\left\{\begin{array}{ll}
[e_i^t,e_j^t]=\sum\limits_{p=k_t+1}^{n_1}\gamma_{i,j}^{t,p}e_t,& 1\leq i, j\leq n_t,\ 1\leq t\leq s,\\[1mm]
[e_i^t,x_i^t]=e_i^t,& 1\leq i\leq k_t,\ 1\leq t\leq s,\\[1mm]
[x_i^t,e_i^t]=-e_i^t,&1\leq i\leq k_t^1,\ 1\leq t\leq s,\\[1mm]
[e_i^t,x_j^t]=\alpha_{i,j}^te_i^t,& k_t\leq i\leq n_t,\  1\leq j\leq k_t, \ 1\leq t\leq s,\\[1mm]
[x_j^t,e_i^t]=\sum\limits_{l=1}^{q}\beta_{t,j,i}^{i_l}e_{i_l}^t, &  k_t+1\leq i\leq n_t, \ 1\leq j\leq k_t,\ 1\leq t\leq s,\\[1mm]
[f_i,y_i]=f_i,& 1\leq i\leq p,\\[1mm]
[y_i,f_i]=(b_i-1)f_i,& b_i\in \{0,1\},\ 1\leq i\leq p;\\[1mm]
\end{array}\right.\end{equation}

\item  $R$ has the form $R=\bigoplus\limits_{t=1}^{s}R_t\oplus R(\mathbb{C}^p),$
with $R_t=N_t\oplus Q_t,$ where $N_t$ is the nilradical of $R_t$ and $Q_t$ is the subspace complementary to $N_t$.
\end{itemize}
\end{thm}

\begin{proof} The proof is immediate from Theorem \ref{thm33} along with a result of \cite{Abelian} (see Theorem 3.2).
\end{proof}

\begin{cor}\label{cor3.7} Any algebra under the condition of Theorem \ref{thm3.13} is isomorphic to one of the following pairwise non isomorphic $(p+1)$ algebras
$$R_1(0,0,\dots,0),\ R_1(1,0,\dots,0), \ R_1(1,1,0,\dots,0),\cdots, R_1(1,1,\dots,1,0),\ R_{1}(1,1\dots,1).$$
\end{cor}

\begin{rem} If the nilradical is a Lie algebra the table of multiplications of $R$ in Theorem \ref{thm3.13} has a simple form as follows
\begin{equation}R_2(b_1,b_2,\dots.b_p):\left\{\begin{array}{ll}
[e_i^t,e_j^t]=\sum\limits_{p=k_t+1}^{n_1}\gamma_{i,j}^{t,p}e_t,& 1\leq i, j\leq n_t,\ 1\leq t\leq s,\\[1mm]
[e_i^t,x_i^t]=-[x_i^t,e_i^t]=e_i^t,& 1\leq i\leq k_t,\ 1\leq t\leq s,\\[1mm]
[e_i^t,x_j^t]=-[x_j^t,e_i^t]=\alpha_{i,j}^te_i^t,& k_t\leq i\leq n_t,\  1\leq j\leq k_t, \ 1\leq t\leq s,\\[1mm]
[f_i,y_i]=f_i,& 1\leq i\leq p,\\[1mm]
[y_i,f_i]=(b_i-1)f_i,& b_i\in \{0,1\},\ 1\leq i\leq p.\\[1mm]
\end{array}\right.\end{equation}
\end{rem}

%

\

\section{Some examples}

In this section we review some results obtained earlier that fit in the case studied in the papar as particular cases.

\begin{exam} Let as the nilradical be taken the nilpotent Leibniz algebra
$$L^1:\quad [e_1,e_1]=e_2, \ [e_i,e_1]=e_{i+1}, \ [e_1,e_{i}]=-e_{i+1}, \ 3\leq i \leq n-1.$$

The generator basis elements of $L_1$ are $e_1, e_3$.
Here is a result of the paper \cite{Shab}.

\begin{thm} \label{thm5.22} There is one solvable indecomposable Leibniz algebra up to isomorphism
with a codimension two nilradical $L^1,(n\geq 4)$, which is given below:
$$q_{n+2,1}: \left\{\begin{array}{lllll}
[e_1,e_1]=e_2, &[e_i,e_1]=e_{i+1}, & [e_1,e_{i}]=-e_{i+1}, & 3\leq i \leq n-1, \\[1mm]
[e_1,e_{n+1}]=e_1,&[e_i,e_{n+1}]=(i-3)e_i,&[e_{n+1},e_1]=-e_1,&[e_{n+1},e_2]=-2e_2,\\[1mm]
[e_{n+1},e_i]=(3-i)e_i,&[e_1,e_{n+2}]=e_1,&[e_i,e_{n+2}]=(i-2)e_i,&[e_{n+2},e_1]=-e_1,\\[1mm]
[e_{n+2},e_2]=-2e_2,&[e_{n+2},e_i]=(2-i)e_i,& \ (3 \leq i \leq n).&\\[1mm]
\end{array}\right.$$
\end{thm}

Applying Theorem \ref{thm35} we get the following $(n+2)$-dimensional solvable Leibniz algebra:
$$R(L_1):\quad \left\{\begin{array}{llll}
[e_1,e_1]=e_2, &[e_i,e_1]=-[e_1,e_{i}]=e_{i+1}, &  3\leq i \leq n-1, \\[1mm]
[e_1,x_1]=e_1,& [x_1,e_1]=(b_1-1)e_1, & b_1 \in \{0,1\},\\[1mm]
[e_3,x_2]=e_3,& [x_2,e_3]=(b_2-1)e_3, & b_2 \in \{0,1\},\\[1mm]
[e_2,x_1]=\alpha_{2,1}e_2, & [e_i,x_1]=\alpha_{i,1}e_i, & 4 \leq i \leq n,&\\[1mm]
[e_2,x_2]=\alpha_{2,2}e_2,& [e_i,x_2]=\alpha_{i,2}e_i, & 4 \leq i \leq n,&\\[1mm]
[x_1,e_2]=\beta_{1,2}e_2,&[x_1,e_i]=\beta_{1,i}e_i,& 4 \leq i \leq n,& \\[1mm]
[x_2,e_2]=\beta_{2,2}e_2,&[x_2,e_i]=\beta_{2,i}e_i,& 4 \leq i \leq n.\\[1mm]
\end{array}\right.$$

The table of multiplications of $L_1$ gives
$$\alpha_{i,1}=i-3,\  4 \leq i \leq n,\quad \alpha_{2,1}=2,\quad \alpha_{2,2}=0,\quad \alpha_{i,2}=1,\  4 \leq i \leq n$$
and $$e_1, e_3, \dots, e_{n-1} \not \in Ann_r(L_1), \quad e_2\in Ann_r(L_1).$$

Since $[e_i,z]+[z,e_i] \in Ann_r(R(L_1))$ for any $z\in R(L_1)$, we get
$$\beta_{1,2}=\beta_{2,2}=0,\quad \beta_{1,i}=3-i, \quad \beta_{2,i}=-1, \ 4 \leq i \leq n-1.$$

Verifying the Leibniz identity in $R(L_1)$ for the triplets $\{x,e_{n-1},e_1\}$ and $\{y,e_{n-1},e_1\}$ we obtain $\beta_{1,n}=3-n,\ \beta_{2,n}=-1$.

Therefore the table of multiplications of $R(L_1)$ looks like
$$R(L_1):\quad \left\{\begin{array}{llll}
[e_i,e_1]=-[e_1,e_{i}]=e_{i+1}, \ 3\leq i \leq n-1,&[e_1,e_1]=e_2, \\[1mm]
[e_1,x_1]=-[x_1,e_1]=e_1, & [e_2,x_1]=2e_2,& \\[1mm]
[e_i,x_1]=-[x_1,e_i]=(i-3)e_i,\ 4 \leq i \leq n,&[e_3,x_2]=-[x_2,e_3]=e_3,\\[1mm]
[e_i,x_2]=-[x_2,e_i]=e_i,& 4 \leq i \leq n.&\\[1mm]
\end{array}\right.$$

The base change $e_{n+1}'=x_1, \ e_{n+2}'=x_1+x_2$ in $R(L_1)$ leads to the algebra $q_{n+2,1}$.
\end{exam}

\begin{exam} Let us take as the nilradical the following split $2$-generated nilpotent Leibniz algebra
$$F_n^2:\quad [e_1,e_1]=e_3, \quad [e_i,e_1]=e_{i+1}, \ 3\leq i \leq n-1.$$

In the paper \cite{Filiform} the following classification theorem was proved.
\begin{thm} An arbitrary $(n+2)$-dimensional solvable
Leibniz algebra with the nilradical $F_n^2$ is isomorphic to one of the
following non isomorphic algebras:

$R_1:\left\{\begin{array}{llll}
[e_1,e_1]=e_3,&[e_i,e_1]=e_{i+1},& 3\leq i \leq n-1,\\[1mm]
[e_1,x]=e_1, &[x,e_1]=-e_1,\\[1mm]
[e_2,y]=-[y,e_2]=e_2, &[e_i,x]=(i-1)e_i,&3 \leq i \leq n.\\[1mm]
\end{array}\right.$
\end{thm}

$R_2:\left\{\begin{array}{llll}
[e_1,e_1]=e_3,& [e_i,e_1]=e_{i+1},& 3\leq i \leq n-1,\\[1mm]
[e_1,x]=e_1,& [x,e_1]=-e_1,\\[1mm]
[e_2,y]=e_2,& [e_i,x]=(i-1)e_i, & 3 \leq i \leq n.\\[1mm]
\end{array}\right.$

If apply Theorem \ref{thm33} we get the table of multiplications of $R(F_n^2)$ as follows.
$$R(F_n^2):\left\{\begin{array}{lllll}
[e_1,x]=e_1,& [x,e_1]=-e_1,\\[1mm]
[e_2,y]=e_2,& [y,e_2]=(b_2-1)e_2,\ b_2 \in \{0,1\},\\[1mm]
[e_i,x]=(i-1)e_i,& 3 \leq i \leq n, & \\[1mm]
[x,e_i]=\beta_{1,i}e_i, & 3\leq i\leq n,\\[1mm]
[y,e_i]=\beta_{2,i}e_i, & 3\leq i\leq n.\\[1mm]
\end{array}\right.$$

From the table of multiplications of $F_n^2$ we get $e_3, \dots, e_{n} \in Ann_r(F_n^2)$ and $e_1 \not \in Ann_r(F_n^2)$. Since $[e_i,z]+[z,e_i] \in Ann_r(R(F_n^2))$ for any $z\in R(F_n^2)$, we obtain $\beta_{1,i}=\beta_{2,i}=0, \ 3 \leq i \leq n.$

If $b_2=0$, then we obtain $R_1$ and if $b_2=1$, then we get $R_2.$
\end{exam}

\begin{exam} In the paper \cite{4dim} the following classification theorem was proved.

\begin{thm} Let $L$ be a $4$-dimensional solvable Leibniz algebra with $2$-dimensional abelian nilradical. Then $L$ is isomorphic to one of the following pairwise non isomorphic algebras:
$$R_4^1:\left\{\begin{array}{ll}
[e_1,x]=e_1,\\[1mm]
[e_2,y]=e_2,\\[1mm]
\end{array}\right.
R_4^2:\left\{\begin{array}{ll}
[e_1,x]=-[x,e_1]=e_1,\\[1mm]
[e_2,y]=-[y,e_2]=e_2,\\[1mm]
\end{array}\right.
R_4^3:\left\{\begin{array}{ll}
[e_1,x]=e_1,\\[1mm]
[e_2,y]=-[y,e_2]=e_2.\\[1mm]
\end{array}\right.$$
\end{thm}

Applying Theorem \ref{thm35} we derive
$$R(abelian):\quad \left\{\begin{array}{lll}
[e_1,x]=e_1,& [x,e_1]=(b_1-1)e_1,& b_1\in \{0,1\},\\[1mm]
[e_2,y]=e_2,& [y,e_2]=(b_2-1)e_2,& b_2\in \{0,1\}.\\[1mm]
\end{array}\right.$$

The algebras $R_4^i, i=1,2,3$ are obtained from $R(abelian)$ by considering all possible values of $b_1$ and $b_2$ (the cases $b_1=1, b_2=0$ and $b_1=0, b_2=1$ being isomorphic to each other algebras due to the symmetricity of $x$ and $y$).
\end{exam}

\begin{exam} Here is the result of the paper \cite{5dim}.

\begin{prop} \label{prop3.31} Let $L$ be a $5$-dimensional solvable Leibniz algebra whose nilradical is given by the product $[e_2,e_1]=e_3$. Then there exists a basis $\{e_1, e_2, e_3, x_1, x_2\}$ of $L$ such that its table of multiplications has the following form:
$$[e_2,e_1]=e_3,\quad [e_1,x_1]=-[x_1,e_1]=e_1, \quad
[e_2,x_2]=e_2,\quad [e_3,x_1]=e_3,\quad [e_3,x_2]=e_3.$$
\end{prop}

The nilradical $N$ here, with the table of multiplication $[e_2,e_1]=e_3$, is nilpotent and 2-generated. From Theorem \ref{thm33} we receive
$$\left\{\begin{array}{lllll}
[e_1,x_1]=e_1, & [x_1,e_1]=-e_1,& [e_3,x_1]=e_3, & [x_1,e_3]=\beta_{1,3}e_3,\\[1mm]
[e_2,x_2]=e_2, & [x_2,e_2]=(b_2-1)e_2,\ b_2 \in \{0,1\},& [e_3,x_2]=e_3,& [x_2,e_3]=\beta_{2,3}e_3.\\[1mm]
\end{array}\right.$$

The following chain of equalities

$[x_1,e_3]=[x_1,[e_2,e_1]]=[[x_1,e_2],e_1]-[[x_1,e_1],e_2]=0,$

$[x_2,e_3]=[x_2,[e_2,e_1]]=[[x_2,e_2],e_1]-[[x_2,e_1],e_2]=[(b_2-1)e_2,e_1]=(b_2-1)e_3,$\\
imply $[x_1,e_3]=0$ and $[x_2,e_3]=(b_2-1)e_3.$
Then due to $[x_1,e_3]+[e_3,x_1]=e_3$ we conclude that $e_3$ lies in the right annihilator of the solvable algebra, which in its turn implies $[x_2,e_3]=0$, i.e., $b_2=1.$ Thus, we obtain that in Proposition \ref{prop3.31}.
\end{exam}

%

\begin{exam}\label{exam49} Let us consider the nilpotent Leibniz algebra
$$NF_n: \quad [e_i,e_1]=e_{i+1}, \quad 1 \leq i \leq n-1.$$

This is a unique, up to isomorphism, in each dimension one-generated nilpotent algebra with the maximal index of nilpotency.
Here is a result of the paper \cite{Omirov}.

\begin{thm} \label{thm3.33} Let $R$ be a solvable Leibniz algebra whose nilradical is $NF_n$. Then there exists a basis
$\{e_1,e_2, \dots, e_n,x\}$ of the algebra $R$ such that the table of multiplications of $R$ with respect to this basis has the following form:
$$\left\{\begin{array}{lll}
[e_i,e_1]=e_{i+1}, &1 \leq i \leq n-1,\\[1mm]
[x,e_1]=e_1,& [e_i,x]=-ie_i,& 1 \leq i \leq n.\\[1mm]
\end{array}\right.$$
\end{thm}
Thanks to Theorem \ref{thm35} the solvable extension of $NF_n$ has the following table of multiplications:
$$R(NF_n):\quad \left\{\begin{array}{llll}
[e_1,x]=e_1, & [x,e_1]=-e_1,& \\[1mm]
[e_i,x]=ie_i,& 2 \leq i \leq n,& \\[1mm]
[x,e_i]=\beta_{i}e_i,& 2 \leq i \leq n.\\[1mm]
\end{array}\right.$$
Since $e_2, \dots, e_n \in Ann_r(NF_n)$, we conclude that $\beta_{i}=0,\ 2\leq i\leq n.$ Then the change $x'=-x$ leads to that in Theorem \ref{thm3.33}.
\end{exam}

\begin{exam} Let us take as the nilradical the following $n$-dimensional nilpotent Leibniz algebra
$$\mu_3: \quad [e_i,e_1]=e_{i+1},\ 2\leq i\leq n-2k-1,\quad [e_2,f_i]=f_{k+i},\ 1\leq i\leq k.$$

One of the results of \cite{p_filiform} was given as follows.
\begin{thm} \label{thmpfill} An arbitrary $(n+k+2)$-dimensional solvable Leibniz algebra with nilradicla $\mu_3$ admits a basis such that its table of multiplications has the following form:
$$\quad\left\{\begin{array}{ll}
[e_i,e_1]=e_{i+1},&2\leq i\leq n-2k-1,\\[1mm]
[e_2,f_i]=f_{k+i},&1\leq i\leq k,\\[1mm]
[e_1,y_1]=e_1,&\\[1mm]
[e_j,y_1]=(j-2)e_j,&3\leq j\leq n-2k,\\[1mm]
[y_1,e_1]=-e_1,&\\[1mm]
[e_j,y_2]=e_j,&2\leq j\leq n-2k,\\[1mm]
[f_{k+i},y_2]=f_{k+i},&1\leq i\leq k,\\[1mm]
[f_i,x_i]=f_i,&1\leq i\leq k,\\[1mm]
[f_{k+i},x_i]=f_{k+i},&1\leq i,j\leq k,\\[1mm]
[x_i,f_i]=-f_{i},& 1\leq i\leq k,\\[1mm]
\end{array}\right.$$
\end{thm}

Since $\mu_3$ is generated by $\{e_1,e_2,f_1,\dots, f_{k}\}$ and $\{e_3,e_4,\dots,e_{n-2k},f_{k+1},\dots, f_{2k}\}\subset Ann_r(\mu_3)$, from Proposition \ref{cor311} we easily get the algebra of Theorem \ref{thmpfill}.
\end{exam}


\section{The completeness}
In this section we prove that the solvable extensions obtained in the paper are complete, i.e., they have the trivial center and all their derivations are inner.

Here is the proof that their center is trivial.
\begin{prop}\label{prop1} Let $R=N\oplus Q$ be a solvable Leibniz algebra such that $dimQ=dim (N/N^2)$. Then $Center(R)=0.$
\end{prop}
\begin{proof} Let us assume that $Center(R)\neq\{0\}$ and $c$ be a non zero element of the center. Due to Theorem \ref{thm31} $R$ admits
a basis $\{e_1, e_2, \dots, e_n, x_1, x_2, \dots, x_k\}$ such that in $R$ the products (\ref{eq1}) hold true.

Let $c=c_1+c_2$, where $c_1 \in N, \ c_2 \in Q.$

If $0 \neq c_2=\sum\limits_{i=1}^{k}\alpha_i x_i$, then there exists $i_0$, where $1\leq i_0 \leq k$ such that $\alpha_{i_{0}} \neq 0.$
From
$$0=[e_{i_0},c]=[e_{i_0},c_1]+[e_{i_0},\sum\limits_{i=1}^{k}\alpha_i x_i]=\sum\limits_{i=k+1}^{n}(*)e_i+\alpha_{i_0} e_{i_0}$$
we deduce $\alpha_{i_{0}}=0,$ which is a contradiction with the assumption that $\alpha_{i_{0}} \neq 0$. Therefore, in $c=c_1+c_2$ the term $c_2$ has to be zero, i.e., $c=c_1$
and let $c_1=\sum\limits_{i=1}^{n}\alpha_ie_i$. Without loss of generality one can assume that $\alpha_{j_0}\neq0,\ 1\leq j_0\leq n$. Let $e_{j_1},$ where $1\leq j_1 \leq k$, be in the word $e_{j_0}$. Then $$0=[c_1,x_{j_1}]=\left[\sum\limits_{i=1}^{n}\alpha_ie_i,x_{j_1}\right]=\alpha_{j_0,j_1}\alpha_{j_0}e_{j_0}
+\sum\limits_{i=1,\ i\neq j_0}^{n}\alpha_{i,j_1}\alpha_ie_i, \ \mbox{where} \ \alpha_{j_0,j_1}\geq 1\ \mbox{and}\ \alpha_{j_0}=0,$$ this contradicts again with the assumption that $c_1\neq 0$.
\end{proof}

To prove that all derivations of $R$ are inner we make use the following proposition.
\begin{prop}\label{der4} Let $R=N\oplus Q$ be a solvable Leibniz algebra such that $dimQ=dim (N/N^2)$. Then $I=Ann_r(R),$ where $I=id<[x,x] \ | \ x\in R>$.
\end{prop}
\begin{proof} Due to Theorem \ref{thm35} we get $I\subseteq Ann_r(R).$  Let suppose that $Ann_r(R)\varsubsetneq I$, $f\in Ann_r(R)\setminus I$
and set
$$f=\sum\limits_{i=1}^{k}a_ie_i+\sum\limits_{i=k+1}^{n}b_ie_i.$$
Let also assume that $b_{i_0}\neq 0$ for some $i_0,$ where $k+1\leq i_0\leq n$ and let $j_0$ for $1\leq j_0\leq k$ such that $\alpha_{i_0,j_0}\neq 0$.
Consider
$$[f,x_{j_0}]=a_{j_0}e_{j_0}+\alpha_{i_0,j_0}b_{i_0}e_{i_0}+\sum\limits_{i=k+1,\ i\neq i_0}^{n}\alpha_{i,j_0}b_ie_i.$$

Since $[f,x_{j_0}]=[f,x_{j_0}]+[x_{j_0},f] \in I$ we obtain
$$f^1=f-\frac{1}{\alpha_{i_0,j_0}}[f,x_{j_0}]=\sum\limits_{i=1}^{k}a_i^1e_i+\sum\limits_{i=k+1,\ i\neq i_0}^{n}b_i^1e_i \in Ann_r(R)\setminus I.$$

Iterating $(n-k)$-times, we obtain
$f^{n-k}=\sum\limits_{i=1}^{k}a_i^{n-k}e_i\in Ann_r(R)\setminus I.$
Then from the chain of equalities
$$I\ni\left[\sum\limits_{i=1}^{k}x_i,f^{n-k}\right]+\left[f^{n-k},\sum\limits_{i=1}^{k}x_i\right]=
\sum\limits_{i=1}^{k}a_i^{n-k}e_i=f^{n-k}\in Ann_r(R)\setminus I$$
we get a contradiction with the assumption that $Ann_r(R)\varsubsetneq I$.
\end{proof}

\begin{rem} \label{rem15} Since the basis elements $e_i, \ i, \ k+1\leq i\leq n$ are in $N^2$, at least two of the generator basis elements have to be in $e_i$.
Therefore,
%
$\sum\limits_{t=1}^{k}\alpha_{i,t}\geq 2$.
Later on we will repeatedly make use this property of $\alpha_{i,t}$.
\end{rem}

\begin{thm} \label{thm53} Let $R=N\oplus Q$ be a solvable Leibniz algebra such that $dimQ=dim (N/N^2)$. Then all derivations of $R$ are inner.
\end{thm}
\begin{proof}
To prove we have to go through a long straightforward arguments as follows. Let $I(N)$ and $I(R)$ denote the ideals generated by squires in $N$ and $R,$ respectively. Here for a subspace $W$ of a vector space $V$ the notation $\pi_W$ stands for the projection to the subspace $W$. The proof is split into two cases.
\begin{itemize}
\item{\textbf{Case 1.}} Let us suppose that $I(N)=I(R).$ Then $I\subseteq N^2$ and we get $G=R/I=(N+Q)/I=\widetilde{N}+Q$. The quotient Lie algebra $G$ satisfies the condition $dim\widetilde{N}/\widetilde{N}^2=dimQ.$ Due to Theorem
\ref{thmLieH1} all its derivations are inner.

Let take $d\in Der(R)$. Then
$$d=d_{G,G}+d_{G,I}+d_{I,I},\ \mbox{where}\ d_{G,G}: G\rightarrow G,\ \ d_{G,I}: G\rightarrow I,\ \ d_{I,I} : I\rightarrow I.$$
Obviously,
$$d(e_i)=d_{G,G}(e_i)+d_{G,I}(e_i), \quad d(x_i)=d_{G,G}(x_i)+d_{G,I}(x_i),\quad 1\leq i\leq k.$$
Moreover, $d_{G,G}\in Der(G).$

It is true that
$$0=d([x_i,x_j])=[d(x_i),x_j]+[x_i,d(x_j)]=[d_{G,G}(x_i),x_j]+[d_{G,I}(x_i),x_j]$$

\qquad \quad $+[x_i,d_{G,G}(x_j)]=d_{G,G}([x_i,x_j])+[d_{G,I}(x_i),x_j]=[d_{G,I}(x_i),x_j],\
\mbox{for}\ 1\leq i,j\leq k.$\\
Now we use Remark \ref{rem15} and the embedding $I\subseteq N^2=Span\{e_{k+1}, \dots, e_n\}$ to get $d_{G,I}(x_i)=0,$ where $1\leq i\leq k$.

We also note that for $1\leq i,j\leq k$ the following holds true
$$d_{G,G}([e_i,x_j])-[d_{G,G}(e_i),x_j]-[e_i,d_{G,G}(x_j)]=d_{G,G}([e_i,x_j])
-[d_{G,G}(e_i),x_j]$$

\qquad $-\pi_G[e_i,d_{G,G}(x_j)]-\pi_I[e_i,d_{G,G}(x_j)]=-\pi_I[e_i,d_{G,G}(x_j)].$\\
Therefore, the components of $\pi_I[e_i,d_{G,G}(x_j)]$ are expressed via structure constants of the algebra $R$ and the components of $d_{G,G}$.

The equality with $1\leq i\neq j\leq k$
$$0=d([e_i,x_j])=[d(e_i),x_j]+[e_i,d(x_j)]=[d_{G,G}(e_i),x_j]+[d_{G,I}(e_i),x_j]$$
$$+[e_i,d_{G,G}(x_j)]=[d_{G,I}(e_i),x_j]+\pi_I[e_i,d_{G,G}(x_j)]$$
gives
\begin{equation}\label{co9}
[d_{G,I}(e_i),x_j]=-\pi_I[e_i,d_{G,G}(x_j)],\quad 1\leq i\neq j\leq k.
\end{equation}

And the following equality with $1\leq i\leq k$
$$d_{G,G}([e_i,x_i])+d_{G,I}(e_i)=d(e_i)=d([e_i,x_i])=[d(e_i),x_i]+[e_i,d(x_i)]
=[d_{G,G}(e_i),x_i]$$
$$+[d_{G,I}(e_i),x_i]+[e_i,d_{G,G}(x_i)]=d_{G,G}([e_i,x_i])+[d_{G,I}(e_i),x_i]+\pi_I[e_i,d_{G,G}(x_i)]$$
implies
\begin{equation}\label{co10}
d_{G,I}(e_i)-[d_{G,I}(e_i),x_i]=\pi_I[e_i,d_{G,G}(x_i)] \quad 1\leq i\leq k.
\end{equation}

If add up (\ref{co9}) for any values of $j$ and subtract (\ref{co10}) from the result we get
$$\sum\limits_{j=1}^{k}[d_{G,I}(e_i),x_j]-d_{G,I}(e_i)=-\sum\limits_{j=1}^{k}\pi_I[e_i,d_{G,G}(x_j)],\ \ 1\leq i\leq k.$$

Due to Remark \ref{rem15} we conclude that the components of $d_{G,I}(e_i)$ are expressed via structure constants of the algebra $R$ and the components of $d_{G,G}$.

Thus, applying Theorem \ref{thmLieH1}, Proposition \ref{der4} and the arguments above we obtain
$$dimDer(R)\leq dimDer(G)=dimInner(G)=dim G-dim Center(G)$$
$$=dim G=dim R-dimI=dim R - dim Ann_r(R)=dim Inner (R),$$
which completes the proof of the theorem in Case 1.

\item{\textbf{Case 2.}} Alternatively, let now suppose that $I(N)\neq I(R)$. We set $\{e_{k_1+1},\dots,e_{k}\}=I(R)\cap \{e_1,\dots,e_{k}\}$ and
$J=ideal<e_{k_1+1},\dots,e_{k}>=Span\{e_{k_1+1},\dots,e_{k}, e_{n_1+1},e_{n_1+2},\dots,e_{n}\}.$ Then $J\subseteq I(R)\cap N.$

Consider the Leibniz algebra $\widetilde{R}=R/J=(N\oplus Q)/J=\widetilde{N}\oplus \widetilde{Q}\oplus Q_1$,where $\widetilde{N}=\{e_{1},e_{2},\dots,e_{k_1},e_{k+1},\dots,
e_{n_1}\}$, $\widetilde{Q}=\{x_1,x_2,\dots,x_{k_1}\}$ and $Q_1=\{x_{k_1+1},\dots,x_k\}$.

The solvable Leibniz algebra $L=\widetilde{N}\oplus \widetilde{Q}$ satisfies the conditions $I(\widetilde{N})=I(L)$ and $dim\widetilde{Q}=dim(\widetilde{N}/\widetilde{N}^2)=k_1.$
Then applying Case 1 we get $dim Der(L)=dim Inner (L)$ and  $dimInner(R)=dim Der(L)+dimQ_1=dimDer(L)+k-k_1.$

The fact $J\subseteq I$ and $d(I)\subseteq I$ ($d\in Der (R)$) yields $d(J)\subset I$.
For $d\in Der (R)$ we get a decomposition
$$d=d_{L,L}+d_{L,J}+d_{L,Q_1}+d_{J,I}+d_{Q_1,N}+d_{Q_1,Q},$$
where
$$d_{L,L}: L \rightarrow L,\quad d_{L,J}: L \rightarrow J,\quad d_{L,Q_1}: L \rightarrow Q_1,$$
$$\quad d_{J,I} : J\rightarrow I, \quad d_{Q_1,N}: Q_1\rightarrow N, \quad d_{Q_1,Q}: Q_1\rightarrow Q.$$
Then

$d(e_i)=d_{L,L}(e_i)+d_{L,J}(e_i)+d_{L,Q_1}(e_i),\quad d(x_i)=d_{L,L}(x_i)+d_{L,J}(x_i)+d_{L,Q_1}(x_i),\ 1\leq i\leq k_1,$

$d(e_i)=\sum\limits_{t=k_1+1}^{k}\theta_{i,t}e_t+d_{J,N^2}(e_i), \quad d(x_i)=d_{Q_1,N}(x_i)+d_{Q_1,Q}(x_i), \ k_1+1\leq i\leq k.$\\
Evidently, $d_{L,L} \in Der(L).$

Taking into account Remark \ref{rem15},
the embedding
$J\subseteq Ann_r(R)$
and
$$0=d([x_i,x_j])=[d(x_i),x_j]+[x_i,d(x_j)]=[d_{Q_1,N}(x_i),x_j],$$
with $k_1+1 \leq i\leq k, \ 1\leq j\leq k,$
we derive
\begin{equation}\label{1222}
  d_{Q_1,N}(x_i)=0,\ k_1+1 \leq i\leq k.
\end{equation}

Consider the chain of equalities
$$0=d([x_i,x_j])=[d(x_i),x_j]+[x_i,d(x_j)]=[d_{L,L}(x_i),x_j]+[d_{L,J}(x_i),x_j]$$
$$+[x_i,d_{L,L}(x_j)]=d_{L,L}([x_i,x_j])+[d_{L,J}(x_i),x_j]=[d_{L,J}(x_i),x_j].$$
i.e., we get
\begin{equation}\label{eq1234}
[d_{L,J}(x_i),x_j]=0.
\end{equation}
particularly,
\begin{equation}\label{eq12345}
  [d_{L,J}(x_i),x_j]=0,\ \mbox{for}\ 1\leq i,j\leq k_1.
\end{equation}
Now we make use (\ref{1222}) in (\ref{eq1234}) with $1 \leq i\leq k_1, \ k_1+1\leq j\leq k$ to obtain
$[d_{L,J}(x_i),x_j]=0,\ 1\leq i\leq k_1,\ k_1+1\leq j\leq k,$ and that along with (\ref{eq12345}) gives $[d_{L,J}(x_i),x_j]=0,\ 1\leq i\leq k_1,\ 1\leq j\leq k.$
Hence,
$$d_{L,J}(x_i)=0,\ 1\leq i\leq k_1 \ \ \mbox{and} \ \ d(x_i)=d_{L,L}(x_i)+d_{L,Q_1}(x_i),\ 1\leq i\leq k_1.$$
Observe that for $1\leq i,j\leq k_1$ we have
\begin{equation}\label{der5}\begin{array}{c}
d_{L,L}([e_i,x_j])-[d_{L,L}(e_i),x_j]-[e_i,d_{L,L}(x_j)]=d_{L,L}([e_i,x_j])-[d_{L,L}(e_i),x_j]-\\[1mm]
\pi_L[e_i,d_{L,L}(x_j)]-\pi_J[e_i,d_{L,L}(x_j)]=-\pi_J[e_i,d_{L,L}(x_j)],\\[1mm]
\end{array}\end{equation}
Considering
$$0=d([e_i,x_j])=[d(e_i),x_j]+[e_i,d(x_j)]=[d_{L,L}(e_i),x_j]+[d_{L,J}(e_i),x_j]
+[e_i,d_{L,L}(x_j)]$$
 \quad $=d_{L,L}([e_i,x_j])+\pi_{J}([e_i,d_{L,L}(x_j)])+[d_{L,J}(e_i),x_j]=
\pi_{J}([e_i,d_{L,L}(x_j)])+[d_{L,J}(e_i),x_j]$\\ for $1\leq i\neq j\leq k_1$ we obtain
\begin{equation}\label{co4}
[d_{L,J}(e_i),x_j]=-\pi_{J}([e_i,d_{L,L}(x_j)]),\ 1\leq i\neq
j\leq k_1.
\end{equation}
Note that for $1 \leq i\leq k_1$ the following takes place
$$d(e_i)=d([e_i,x_i])=[d(e_i),x_i]+[e_i,d(x_i)]=[d_{L,L}(e_i),x_i]+[d_{L,J}(e_i),x_i]$$
$$+[e_i,d_{L,L}(x_i)]=d_{L,L}(e_i)+\pi_{J}([e_i,d_{L,L}(x_i)])+[d_{L,J}(e_i),x_i].$$
On the other hand

$\qquad d(e_i)=d([e_i,x_i])=d_{L,L}(e_i)+d_{L,J}(e_i)+d_{L,Q_1}(e_i).$\\
Hence,
\begin{equation}\label{co5}
d_{L,Q_1}(e_i)=0,\quad
d_{L,J}(e_i)-[d_{L,J}(e_i),x_i]=\pi_{J}([e_i,d_{L,L}(x_i)]),\
1\leq i\leq k_1.
\end{equation}
Due to (\ref{co5}) and (\ref{1222}) for $1\leq i\leq k_1,\ k_1+1\leq j\leq k$ we have
$$0=d([e_i,x_j])=[d(e_i),x_j]+[e_i,d(x_j)]=[d_{L,J}(e_i),x_j]+[e_i,d_{Q_1,Q}(x_j)].$$
which implies
\begin{equation}\label{co11}
d_{Q_1,\widetilde{Q}}(x_j)=0,\quad [d_{L,J}(e_i),x_j]=0,\ 1\leq i\leq k_1,\
k_1+1\leq j\leq k.
\end{equation}
Like in Case 1 from (\ref{co4}), (\ref{co5}) and (\ref{co11}) we derive
$$d_{L,J}(e_i)-\sum\limits_{j=1}^{k}[d_{L,J}(e_i),x_j]=\sum\limits_{j=1}^{k_1}\pi_J([e_i,d_{L,L}(x_j)]),\ 1 \leq i\leq k_1$$
and which gives that the components $d_{L,J}(e_i)$ are expressed via structure constants of the algebra
$R$ and the components of $d_{L,L}$.

The chain of the equalities with $k_1+1\leq i\leq k$:

$\sum\limits_{t=k_1+1}^{k}\theta_{i,t}e_t+d_{J,N^2}(e_i)=d(e_i)=d([e_i,x_i])
=[d(e_i),x_i]+[e_i,d(x_i)]$

$\left[\sum\limits_{t=k_1+1}^{k}\theta_{i,t}e_t+d_{J,N^2}(e_i),x_i\right]+[e_i,d_{Q_1,Q_1}(x_i)]=
\theta_{i,i}e_i+[d_{J,N^2}(e_i),x_i]+[e_i,d_{Q_1,Q_1}(x_i)],$\\
give
\begin{equation}\label{co6}\begin{array}{l}
[e_i,d_{Q_1,Q_1}(x_i)]=0,\quad \theta_{i,t}=0, \\ d_{J,N^2}(e_i)-[d_{J,N^2}(e_i),x_i]=0, \ \ k_1+1\leq i\neq t\leq k.\\[1mm]
\end{array}\end{equation}

For $k_1+1\leq i\neq j\leq k$ we have

$0=d([e_i,x_j])=[d(e_i),x_j]+[e_i,d(x_j)]=[\theta_{i,i}e_i+d_{J,N^2}(e_i),x_j]$

\qquad$[e_i,d_{Q_1,Q_1}(x_j)] =[d_{J,N^2}(e_i),x_j]+[e_i,d_{Q_1,Q_1}(x_j)].$

Therefore,
\begin{equation}\label{co7}
[d_{J,N^2}(e_i),x_j]=0, \quad \ k_1+1\leq i\neq j\leq k
\end{equation}
and
\begin{equation}\label{co77}
[e_i,d_{Q_1,Q_1}(x_j)]=0, \quad \ k_1+1\leq i\neq j\leq k.
\end{equation}

From (\ref{co6}), (\ref{co77}) we obtain $d_{Q_1,Q_1}(x_i)=0,\ k_1+1\leq i\leq k,$ that means
$$d(x_i)=0,\ k_1+1\leq i\leq k.$$

Since $d(x_j)=d_{L,L}(x_j)+d_{L,Q_1}(x_j)$ for $k_1+1\leq i\leq k,\ 1\leq j\leq k_1$ we have

$0=d([e_i,x_j])=[d(e_i),x_j]+[e_i,d(x_j)]$

\quad$=[\theta_{i,i}e_i+d_{J,N^2}(e_i),x_j]+[e_i,d_{L,L}(x_j)]
+[e_i,d_{L,Q_1}(x_j)]$

\quad$=[d_{J,N^2}(e_i),x_j]+[e_i,d_{L,L}(x_j)]+[e_i,d_{L,Q_1}(x_j)].$\\
This implies  $[e_i,d_{L,Q_1}(x_j)]=0,$ and therefore, $d_{L,Q_1}(x_j)=0$, that is, $$d(x_j)=d_{L,L}(x_j),\ 1\leq j\leq k_1$$
and
\begin{equation}\label{co12}\begin{array}{c}
[d_{J,N^2}(e_i),x_j]=-[e_i,d_{L,L}(x_j)],\ k_1+1\leq i\leq k,\ 1\leq j\leq k_1.\\[1mm]
\end{array}\end{equation}

Adding up the equalities (\ref{co7}) and (\ref{co12}) for possible values of $j$, then subtracting (\ref{co6}) from the result we obtain
$$d_{J,N^2}(e_i)-\sum\limits_{j=1}^{k}[d_{J,N^2}(e_i),x_j]=\sum\limits_{j=1}^{k_1}[e_i,d_{L,L}(x_j)],\ k_1+1\leq i\leq k.$$
Due to Remark \ref{rem15} we conclude that the components of
$d_{J,N^2}(e_i),\  k_1+1\leq i\leq k$ are expressed via structure
constants of the algebra $R$ and the components of
$d_{L,L}$.

Finally, applying the result of Case 1 and the relations obtained above we figure out the following argument
$$dimDer(R)\leq dimDer(L)+k-k_1=dim Inner(L)+k-k_1=dimInner(R),$$
which completes the proof of the theorem.
\end{itemize}
\end{proof}


Note that the completeness of the algebra $R(NF_n)$ from Example \ref{exam49} has been proved earlier in \cite{Ancochea} (Proposition 1). But now this immediately follows from Theorem \ref{thm53}.


In fact, earlier the notion of a complete (left) Leibniz algebra has been introduced in \cite{Ernie}. Since the present paper deals with the right Leibniz algebras the adapted version of the completeness for the right Leibniz algebras is as follows.

\begin{defn} \label{defnErnie} A Leibniz algebra $L$ is said to be complete if
\begin{itemize}
\item[(1)] $Center(L/I) = {0}$;
\item[(2)] for every derivation $d\in Der(L)$ there exists $x\in L$ such that $Im(d-R_x)\subseteq I$,
\end{itemize}
where $I=id<[x,x] \ | \ x\in L>$.
\end{defn}

Note that a complete Leibniz algebra in the sense of Definition \ref{defncomplete} also is complete in the sense of Definition \ref{defnErnie} and these two notions are not equivalent. Indeed, by Remark in \cite{Ernie} a semisimple Leibniz algebra is complete in the sense of Definition \ref{defnErnie}, whereas a semisimple Leibniz algebra need not be complete in the sense of Definition \ref{defncomplete} (see Example 3.4 in \cite{Ernie}). So, it seems that completeness in the sense of \cite{Ernie} is more adequate. But another classical result on the existence of outer derivations of nilpotent Lie algebra implies that nilpotent Lie algebras are not complete. This result is not true for nilpotent Leibniz algebra in the sense of completeness given in \cite{Ernie} (see Example 3.8 in \cite{Ernie}). We anticipate this result for nilpotent Leibniz algebras under Definition \ref{defncomplete}.

\end{document}